\documentclass{amsart}
\usepackage{dino}

\title{Algorithmic aspects of left-orderings of solvable Baumslag--Solitar
groups via its dynamical realization}

\thanks{Part of this research was performed during Summer 2022 while Ho and Rossegger were visiting the Mathematical Sciences Research Institute (MSRI), now becoming the Simons Laufer Mathematical Sciences Institute (SLMath), which is supported by the National Science Foundation (Grant \#DMS-1928930).
Ho and Rossegger wish to thank the American Institute of Mathematics for stimulating this
research and providing the venue for the discussion of part of this work.
The work of Ho was partially supported by the National Science Foundation under Grant \#DMS-2054558.
The work of Le was partially supported by the AMS-Simons Travel Grant.
The work of Rossegger was supported by the European Union's Horizon 2020 Research and Innovation Programme under the Marie Sk\l{}odowska-Curie grant agreement No. 101026834 — ACOSE.}
\author{Meng-Che ``Turbo'' Ho}
\address{Department of Mathematics, California State University,
Northridge}
\email{turbo.ho@csun.edu} 
\urladdr{https://sites.google.com/view/turboho}
\author{Khanh Le}
\address{Department of Mathematics, Rice University}
\email{khanh.le@rice.edu}
\urladdr{https://sites.google.com/view/khanhqle}

\author{Dino Rossegger}
\address{Institute für Diskrete Mathematik und Geometrie, Technische Universit\"at
Wien}
\email{dino.rossegger@tuwien.ac.at}
\urladdr{https://drossegger.github.io}

\newcommand{\Elo}{\mathrel{E_{lo}}}

\newcommand{\E}{\mathrel{E}}
\newcommand{\Q}{\mathbb{Q}}
\newcommand{\R}{\mathbb{R}}
\newcommand{\Z}{\mathbb{Z}}
\newcommand{\N}{\mathbb{N}}

\DeclareMathOperator{\BS}{BS}
\newcommand{\Elobs}{\mathrel{E_{lo}^{\mathrm{BS}(1,n)}}}

\DeclareMathOperator{\Homeo}{Homeo}
\DeclareMathOperator{\Aff}{Aff}
\DeclareMathOperator{\Stab}{Stab}
\DeclareMathOperator{\Aut}{Aut}

\begin{document}

\maketitle
\begin{abstract}
  We answer a question of Calderoni and Clay~\cite{calderoni2023a} by showing
  that the conjugation equivalence relation of left orderings of the Baumslag-Solitar groups
  $\BS(1,n)$ is hyperfinite for any $n$. Our proof relies on a classification of
  $\BS(1,n)$'s left-orderings via its one-dimensional dynamical
  realizations. We furthermore use the effectiveness of the dynamical realizations of
  $\BS(1,n)$ to study algorithmic properties of the left-orderings on
  $\BS(1,n)$.
\end{abstract}

\section{Introduction}
A group $G$ is \emph{left-orderable} if there is a linear ordering $\prec$ on the
elements of $G$ such that
for all $f,g,h\in G$, $g\prec h$ implies $fg\prec fh$. We refer to such a linear ordering as a \emph{left-ordering} of $G$. The study of (left-)orderable groups has a long tradition in mathematics starting with the work of Dedekind and Hölder in the late 19th and early 20th century. Dedekind famously characterized the real numbers as a complete bi-orderable Abelian group and Hölder showed that any Archimedean ordered group is isomorphic to an additive subgroup of the reals with their standard ordering. These fundamental results led to an influx of interest in orderable groups and established their theory as a cornerstone of group theory; see~\cite{kopytov1996} for a treatment of the classical theory. 

While most studies of orderable groups employed algebraic methods, there is a strong connection with one-dimensional dynamics. Indeed, a group is left-orderable if and only if it acts faithfully on the real line by orientation preserving homeomorphisms~\cite{ghys2001} (see also \cref{sec:dynreal}). 
Motivated by this observation, Navas~\cite{navas2010} systematically applied dynamical ideas to study orderable groups and give new proofs of results previously obtained by algebraic methods, as well as new results. 
He gave a new dynamical proof of the fact that $\LO(F_n)$, the space of left-orderings of the non-abelian free group of rank $n$, is homeomorphic to the Cantor set if $n\geq 2$
\cite[Theorem A]{navas2010}, and of the result of Linnell which states that the number of left-orderings of a group is either finite or uncountable \cite[Theorem C]{navas2010}. Since then one-dimensional dynamics has become an important tool in orderable group theory with many applications. Generalizing Navas's approach, Rivas proved that for all left-orderable groups $G$ and $H$, the free-product $G*H$ has no isolated left-orderings \cite[Theorem A]{rivas2012}. In~\cite[Theorem 1.1]{MMRT2019} dynamics was used to find new examples of groups with isolated left-orderings, and in~\cite{rivas2010,RivasRomain16}  characterizations of left-orderings of various solvable groups were obtained.

These developments have led to various natural questions about the space of left-orderings of groups, $\LO(G)$. Of particular interest to this paper is a question by Deroin, Navas, and Rivas~\cite[Question 2.2.11]{deroin2016} that asks if the conjugation equivalence relation of $G$ on $\LO(G)$ is standard. This question has attracted the interest of Calderoni and Clay~\cite{calderoni2022} who initiated the study of the conjugation equivalence of orderings on a fixed group $G$ in the setting of descriptive set theory. 

Among other groups, Calderoni and Clay studied the space of linear orders of solvable Baumslag-Solitar groups \cite{calderoni2023a}. Baumslag-Solitar groups are introduced in \cite{baumslag1962} as an example of non-Hopfian groups, and have served as important examples and counterexamples in group theory \cite[Chapter 5]{bonanome18}. In particular, the solvable Baumslag-Solitar groups admit nice structural properties and thus provide useful test cases for theories and techniques. 

The main contribution of our paper is in this context: We show that the conjugation equivalence relation of the solvable Baumslag-Solitar group $BS(1,n)$ is hyperfinite for every $n$. This answers a question posed by Calderoni and Clay~\cite{calderoni2023a}. 

The algorithmic aspects of left-orderable groups have also seen attention in the past, mainly focusing on the complexity of orderings of computable groups~\cite{downey1986,harrison-trainor2018,darbinyan2020,darbinyan2022} and their reverse mathematics~\cite{solomon2001,solomon2002}. In \cref{sec:computability}, we explore how a group's dynamics can be used to study algorithmic properties of its orderings, using
$\BS(1,n)$ as an example. Our main result shows that the index sets of orderings that are conjugates of a given ordering with irrational base point is $\Sigma^0_3$-complete. Our proof relies on the dynamical realizations of $\BS(1,n)$ and the machinery developed in prior sections.

Before we prove the main results of this paper we review the main tools used in their proofs and give the necessary definitions to formally state our results.
\section{Left-orderable groups, their dynamical realizations, and $\Elo$}
A left ordering $\prec$ on a group $G$ induces a partition of $G$ into disjoint subsets 
\[P^+=\{g \in G \mid  g\succ id\},
P^{-}=(P^{+})^{-1}=\{g \in G \mid id\succ g\} \text{ and } \{id\}\] 
where $P^+$ is called the \emph{positive cone} of the left-ordering $\prec$. Notice that the reverse order $\prec^*$, defined as $g \prec^*h \text{ if and only if } h \prec g$, is also a left-ordering of $G$ with associated positive cone $P^-$. It
is not hard to see that the positive cones on $G$ are precisely the subsets $P\subseteq G$ satisfying
\[ P\cap P^{-1}=\varnothing,\, PP\subseteq P \text{ and } P\cup P^{-1}\cup
\{id\}=G.\]
Moreover, every positive cone gives rise to an associated left-ordering
$\prec_P$ via 
\[ g\prec_P h \iff g^{-1}h\in P\] 
and thus we get a bijection between positive cones and left-orderings on
$G$. 

The collection of all positive cones $P$ of $G$ forms a closed subspace, $\LO(G)$, via the subspace topology of $2^G$ and is thus a Polish
space \cite{sikora2004}. Given any positive cone $P\in \LO(G)$ and any element $g\in G$, the set $P^g=\{ p^g=g^{-1}pg: p\in P \}$ defines a positive cone on $G$. Consequently, the group $G$ acts on $\LO(G)$ via conjugation simply by defining $g(P) = P^{g^{-1}}$ for all $g\in G$ and $P \in \LO(G)$. 
It is not hard to see that the action of $G$ on $\LO(G)$ is continuous and, in fact, computable uniformly in $G$.
\begin{remark}
A countable group $G$ is \emph{computable} if its domain and group operation are computable. We can assume that the domain of $G$ is all of $\omega$ and thus view positive cones $P\subset G$ as subsets of the natural numbers and $2^G$ simply as $2^\omega$. Then the above comment that the action of $G$ on $\LO(G)$ is \emph{computable uniformly in $G$} means that there is a Turing operator $\Phi$ such that
\[ \Phi(G,P;g)=g(P) \text{ for all left-orderable groups $G$, positive cones $P$, and $g\in G$.} \]
\end{remark}


\subsection{Dynamical realizations}\label{sec:dynreal}
Although left-orderability is an algebraic concept, it has a deep connection to one-dimensional dynamics. In particular, the left-orderable countable group can be characterized in dynamical terms.
\begin{theorem}[{\cite[Theorem 6.8]{ghys2001}}]\label{thm:ghys}
  Let $G$ be a countable group. Then the following are equivalent:
  \begin{enumerate}
    \item $G$ is left-orderable.
    \item $G$ acts faithfully on the real line by orientation preserving homeomorphisms, i.e., there is a faithful representation $\gamma: G \to
    \mathrm{Homeo}_+(\mathbb R)$.
  \end{enumerate}
\end{theorem}

Let us elaborate on \cref{thm:ghys}. Given an embedding $D$ of $G$ into $\mathrm{Homeo}_+(\mathbb R)$ and a dense sequence $(x_1,\dots)$
in $\mathbb R$, we can
obtain a positive cone $P_D = P_D(x_1, \dots)$ as follows: we define $g\in P_D$ if, for the least $i$ such that
$D(g)(x_i)$ is not a fixed point, $D(g)(x_i)>x_i$. The proof of the reverse implication that $G$ is left-orderable implies that $G$ embeds into $\mathrm{Homeo}_+(\mathbb R)$ is effective. In particular, given a left-ordering on $G$, there is an associated group action of $G$ on the real line, called a \emph{dynamical realization} of $G$, constructed as follows. Fix a left-ordering $\prec$ on $G$ and, since $G$ is countable, fix an enumeration of the elements of $G = (g_0,g_1,\dots)$. We define a map $t:G\to \mathbb{R}$ that preserves $\prec$, namely,
\[
t(g) < t(h) \iff g \prec h, 
\]
by defining $t:G\to \mathbb R$ inductively starting with $t(g_0)=0$ and 
\[ t(g_i)=\begin{cases}
  \max\{t(g_0),\dots t(g_{i-1})\}+1&\text{ if $(\forall j<i)g_{j} \prec g_i$}\\
  \min\{t(g_0),\dots t(g_{i-1})\}-1&\text{ if $(\forall j<i)g_{i} \prec g_j$}\\
  \frac{t(g_m)+t(g_n)}{2} & \text{ if $g_i\in (g_m,g_n)$, $m,n<i$ and $(\forall j<i)g_j\not\in (g_m,g_n)$}
\end{cases}\]
Then we can define an action of $G$ on $t(G)$ via $g(t(g_i))=t(gg_i)$ and extend this action continuously to the closure $\overline{t(G)}$. Finally, we extend the action to $\mathbb{R} \setminus \overline{t(G)}$ by affine maps to obtain a faithful orientation-preserving action of $G$ on $\mathbb R$ and a faithful representation $D: G \to \Homeo^+(\mathbb{R})$. By construction, this action 
\begin{itemize}
    \item has no global fixed point unless $G$ is the trivial group, and
    \item the orbit of $0$ is free; i.e., the stabilizer of $0$ under this action is trivial.
\end{itemize}
These two properties characterize the dynamical realization up to topological semiconjugacy. In particular, we have

\begin{lemma}[{\cite[Lemma 2.8]{navas2010}}]
    Let $\prec_1$ be a left-ordering on a non-trivial countable group $G$, and let $D_1$ be a dynamical realization of $\prec_1$. Let $D_2$ be an action of $G$ on $\mathbb{R}$ by orientation-preserving homeomorphisms such that 
    \begin{itemize}
        \item $D_2$ has no global fixed point, and
        \item the orbit of $0$ is free. 
    \end{itemize}
    If $\prec_2$ is a left-ordering on $G$ defined by 
    \[
    g \prec_2 h \Leftrightarrow D_2(g)(0) < D_2(h)(0),
    \]
    then the two left-orderings $\prec_1$ and $\prec_2$ coincide if and only if $D_2$ is \emph{topologically semiconjugate} to $D_1$ relative to $0$. That is, there exists a continuous non-decreasing surjective map $\varphi:\mathbb{R} \to \mathbb{R}$ such that $D_1(g)\circ \varphi = \varphi \circ D_2(g)$ for all $g \in G$ and that $\varphi(0) = 0$ \footnote{The condition that $\varphi$ fixes the origin was not stated in \cite{navas2010}, but was used in the proof as pointed out in \cite[Lemma 3.7]{mannrivas2018}}.
\end{lemma}

\begin{remark}
    Topological semiconjugacy is not an equivalence relation as $\varphi$ could collapse some intervals to points. The lemma above says that among all actions satisfying the two conditions, the dynamical realization is ``minimal''.
\end{remark}

It follows that a different choice of enumeration of $G$ yields a different action that is topologically semiconjugate to $D$. Therefore, we may speak of the dynamical realization of a left-ordering on $G$ without referencing the enumeration of $G$. From the dynamical realization, we can also recover the original positive cone $P$ on $G$ and its conjugate.

\begin{proposition}
    Suppose that $D$ is a dynamical realization of a left-ordering $P$ of $G$. Let $t :G \to \mathbb{R}$ be the order-preserving map used in the construction of $D$. Then for any $h \in G$, we have
    \[
    P^h = \{g \in G: D(g)(t(h^{-1})) > t(h^{-1})\}.
    \]
\end{proposition}

\begin{proof}
    When $h = id$, we have
    \[
    P = \{g \in G: D(g)(t(id)) > t(id)\}
    \]
    since $t$ is an order-preserving map. In general, we have
    \[
    \begin{aligned}
    P^h &= \{h^{-1}gh: g \in P\} \\
        &= \{h^{-1}gh: D(g)(t(id)) > t(id) \} \\
        &= \{f \in G: D(hfh^{-1})(t(id)) > t(id) \} \\
        &= \{f \in G: D(f)(t(h^{-1})) > t(h^{-1}) \}. 
    \end{aligned}
    \]
\end{proof}

An application of the dynamical realization that is useful for this paper is an effective classification of the left-orderings on the solvable Baumslag--Solitar group $\BS(1,n)$ \cite[Theorem 4.2]{rivas2010}. We will review this classification in \cref{sec:hypfinite}.

\subsection{Descriptive set theory and $\Elo$}

For a fixed group $G$, the conjugation action of $G$ on $\LO(G)$
defines an orbit equivalence relation, denoted $\E_{lo}^G$, where
 \[P\E^G_{lo} Q \Leftrightarrow \exists g\in G, g(P)=Q.\]
When $G$ is countable, every equivalence class of $\E_{lo}^G$ (or just $\Elo$ if $G$ is clear from
context) is countable. 

Equivalence relations where each equivalence class is countable are called
\emph{countable equivalence relations} and are a major topic in descriptive set
theory, where the complexity of equivalence relations is measured using Borel reducibility $\leq_B$. The structure of the
quasi-order of countable Borel equivalence relations under Borel reducibility is complicated and its
investigation is an active research area,
see~\cite{kechris2019} for an overview of developments. Let us
mention three benchmark equivalence relations:
\begin{enumerate}
  \item The \emph{identity relation} on $2^\omega$, $id^{2^\omega}$, is the least
  complicated equivalence relation among countable equivalence relations on
  uncountable spaces. The equivalence relations reducible to $id^{2^\omega}$ are
  called \emph{smooth}.
  \item The equivalence relation of eventual equality on $2^\omega$, 
    \[ x\E_0 y \iff \exists m (\forall n>m) x(n)=y(n)\]
    is the archetypical non-smooth \emph{hyperfinite} equivalence relation (i.e., an
    increasing union of equivalence relations having only finite classes). By
    \cite{harrington1990} every
    hyperfinite equivalence relation is either bi-reducible with $\E_0$ or
    smooth.
\item The orbit equivalence relation $\mathbin S$ of the shift action of $F_2$ on $2^{F_2}$ is \emph{universal} for countable Borel 
  equivalence relations, i.e., every other countable Borel equivalence relation
  is Borel reducible to it.
\end{enumerate}
While the interval $(id^{2^\omega}, \E_0)$ is trivial, the interval between
$\E_0$ and $S$ is known to be extremely complicated.

A fundamental result due to Feldman and Moore~\cite{feldman1977} shows that the
countable equivalence relations are precisely the orbit equivalence relations of
Borel actions of countable groups. One major conjecture in the area, known as
Weiss's conjecture, aims to shed light on which groups cannot have complicated
orbit equivalence relations. It states that every Borel action of an amenable
group has a hyperfinite orbit equivalence relation.
So far this conjecture has not been fully confirmed and only partial results are
known with the latest advance made in~\cite{conley2023}. 

Deroin, Navas, and Rivas \cite{deroin2016} asked if $\LO(G)$ modulo the action of $G$ is standard, which is equivalent to asking if $\E_{lo}^G$ is smooth. Calderoni and Clay generalized this to studying the complexity of equivalence relations $\E_{lo}^G$ under Borel reductions \cite{calderoni2022,calderoni2023a,calderoni2023}. They showed that $\E_{lo}$ is universal for free products of countable left-orderable groups and not smooth for many groups, including the Baumslag--Solitar group $\BS(1,n)$ and the Thompson's group $F$. They also showed that $E_{lo}^{\BS(1,2)}$ is hyperfinite. It is still open if there is a group $G$ with $\E_{lo}^G$ being \emph{intermediate}, namely, strictly between $E_0$ and $S$. There are also other closely related orbit equivalence relations, including the action of $\Aut(G)$ on the Archimedean ordering of $\Z^n$ or $\Q^n$. In \cite[Theorem 1.1]{CMMS2023}, it was shown that the orbit equivalence relation of $\Aut(\mathbb{Q}^2)$ on the space of Archimedean orderings of $\mathbb{Q}^2$ is not smooth. Extending this result, Poulin showed that the action of $\Aut(\mathbb{Q}^n)$ on $\LO(\mathbb{Q}^n)$ is not hyperfinite when $n\geq 3$ \cite[Corollary 1.3]{poulin2024}.

In the next section, we use the affine action of $\BS(1,n)$ on $\mathbb{R}$ to show that the $\Elobs$
is hyperfinite for every $n$ (Theorem \ref{thm:hypfinite}), answering a question of
Calderoni and Clay~\cite[Question 4.2] {calderoni2023}. 
While this result does follow from
Weiss's conjecture, the conjecture has not been confirmed for $\BS(1,n)$, $n>2$. In any case, we believe that our proof is of interest as it is quite elementary compared to known proofs of parts of Weiss's conjecture. Furthermore, the result determines exactly the complexity of the conjugation action of $\BS(1,n)$ on $
\LO(\BS(1,n))$ which is of interest from the orderable group theory perspective. 


\section{$\Elobs$ is hyperfinite}\label{sec:hypfinite}



The solvable Baumslag--Solitar group $\BS(1,n)$, given by the presentation
\[ \langle a,b: b^{-1}ab=a^n\rangle,\]
is an important example in group theory. 
The normal closure $\langle \langle a \rangle \rangle$ of $a$ is abstractly isomorphic to $\mathbb{Z}[1/n]$, the subgroup of $\Q$ generated by $1/n, 1/n^2, 1/n^3, \dots$, via an isomorphism sending $a$ to $1$ and $b^{k}ab^{-k}$ to $1/n^k$ for every $k \in \Z$. Abusing notation, we will write elements of $\langle\langle a \rangle\rangle$ as $a^r$ where $r \in \mathbb{Z}[1/n].$ The quotient of $\BS(1,n)$ by the normal closure of $a$ is the infinite cyclic group generated by the image of $b$. Therefore, $\BS(1,n)$ fits into a (split) short exact sequence
\[
0 \to \mathbb{Z}[1/n] \to \BS(1,n) \to \mathbb{Z} \to 0,
\]
and admits the semidirect product structure $\BS(1,n) = \langle\langle a \rangle\rangle \rtimes \langle b \rangle$. The elements of $G$ have the normal forms $a^rb^s$ where $r \in \mathbb{Z}[1/n]$ and $s \in \mathbb{Z}$. As a warm-up, we first recall a well-known construction of left-orderings using a short exact sequence. 

\begin{proposition}
\label{prop:OrderingFromSES}
    Let $K$ and $H$ be left-orderable groups equipped with positive cones $P_K \subset K$ and $P_H \subset H$. Consider the following short exact sequence of groups:
    \[
    1 \to K \to G \xrightarrow{\pi} H \to 1
    \]
    The set $P_G := \{ g \in G \mid \pi(g) \in P_H\} \cup P_K$ defines a positive cone of $G$.
\end{proposition}




To order $\BS(1,n)$, we observe that there are exactly two orderings on $\mathbb{Z}[1/n]$ and $\mathbb{Z}$: the ordering coming from the standard ordering on $\mathbb{R}$ and its reversal. Applying \cref{prop:OrderingFromSES}, we get four left-orderings on $\BS(1,n)$. The positive cones of these four left-orderings are:
\begin{itemize}
    \item $P_\infty^{++} = \{a^rb^s:s>0 \lor (s = 0 \land r > 0)\}$,
    \item $P_\infty^{+-} = \{a^rb^s:s>0 \lor (s = 0 \land r < 0)\}$,
    \item $P_\infty^{-+} = \{a^rb^s:s<0 \lor (s = 0 \land r > 0)\}$,  
    \item $P_\infty^{--} = \{a^rb^s:s<0 \lor (s = 0 \land r < 0)\}$,
\end{itemize}
We note that all four left-orderings above are conjugation invariant. In other words, they are all bi-orderings on $\BS(1,n)$ and the action of $\BS(1,n)$ on $\LO(\BS(1,n))$ fixes these four bi-orderings. However, $\BS(1,n)$ admits other left-orderings not coming from Proposition \ref{prop:OrderingFromSES}. To study $\Elobs$, we will review the classification of all left-orderings on $\BS(1,n)$ for any integer $n \geq 2$ due to Rivas~\cite[Theorem 4.2]{rivas2010}. Note that although Rivas stated the classification of left-orderings only for $\BS(1,2)$, his proof works without modification for any positive integer $n$. The proof for arbitrary positive integers $n$ is also presented in~\cite{deroin2016}. 

A different source of left-orderings on $\BS(1,n)$ comes from the affine action of $\BS(1,n)$ on $\mathbb{R}$. Consider the action of $\BS(1,n)$ on $\mathbb{R}$ given by $\rho:\BS(1,n) \to \Aff^+(\mathbb{R})$ where
\[
\rho(a)(x) = x+1 \quad \text{and} \quad \rho(b)(x) = x/n.
\]
It is a straightforward computation that this action is faithful. 

\begin{lemma}\cite[page 10]{rivas2010}
Let $x \in \mathbb{R}$.  If $x \in \mathbb{Q}$, then the stabilizer $\Stab_\rho(x) \cong \mathbb{Z}$. If $x$ is not in $\mathbb{Q}$, then the stabilizer $\Stab_\rho(x)$ is trivial.  
\end{lemma}

\begin{proof}
First, we observe that the stabilizer of any point must be either trivial or isomorphic to $\mathbb{Z}$. Under $\rho$, the normal closure of $a$ acts by translation and has no fixed points. Therefore, if the stabilizer of some point on $\mathbb{R}$ is nontrivial, it must be mapped injectively into the quotient $\BS(1,n)/\langle\langle a \rangle \rangle \cong \mathbb{Z}$. Since $\rho$ is a faithful representation, the nontrivial stabilizer must be isomorphic to $\mathbb{Z}.$

For any $r\in \mathbb{Z}[1/n]$ and $s\in \mathbb{Z}$, we have
\[
\rho(a^rb^s)(x) = n^{-s}x + r.
\]
If $s \neq 0$, then the affine map above has a fixed point which must be rational. Now suppose that $x = p/q \in \mathbb{Q}$. We want to find $r \in \mathbb{Z}[1/n]$ and $s\in \mathbb{Z}$ such that 
\[
n^{s}\frac{p}{q}+r = \frac{p}{q}
\]
Let $q = dq'$ and $p'=np/d$ where $d = \gcd(q,n)$. The previous equation is equivalent to
\[
n^{s} \frac{np}{dq'} + nr = \frac{np}{ dq'} \quad \text{or}\quad (n^{s} -1) \frac{p'}{q'} = - nr 
\]
Since $n$ and $q'$ are relatively prime, $q'$ is in the (multiplicative) group of units of $\Z/n\Z$, so we can find $s\in\mathbb{N}\setminus\{0\}$ such that $q'$ divides $n^s-1$. Now we set
\[
r = -\frac{n^s -1}{q'}\frac{p'}{n} \in \mathbb{Z}[1/n]
\]
since $p' \in \mathbb{Z}$. Therefore $\rho(a^rb^{-s})$ fixes $p/q$. 
\end{proof}

It follows that if $\varepsilon \in \mathbb{R}\setminus\mathbb{Q}$, then each of the following subsets defines a positive cone on $\BS(1,n)$:

\begin{itemize}
    \item $P_\varepsilon^+ = \{g:\rho(g)(\varepsilon) > \varepsilon \}$,
    \item $P_\varepsilon^- = \{g:\rho(g)(\varepsilon) < \varepsilon \}$
\end{itemize}
When $\varepsilon\in \mathbb{Q}$, we define the following four positive cones.
\begin{itemize}
    \item $Q_\varepsilon^{++} = \{g:(\rho(g)(\varepsilon) > \varepsilon) \lor (\rho(g)(\varepsilon) = \varepsilon \land  \rho(g)(\varepsilon+1) > \varepsilon + 1 ) \}$
    \item $Q_\varepsilon^{+-} = \{g:(\rho(g)(\varepsilon) > \varepsilon) \lor (\rho(g)(\varepsilon) = \varepsilon \land  \rho(g)(\varepsilon+1) < \varepsilon + 1 ) \}$
    \item $Q_\varepsilon^{-+} = \{g:(\rho(g)(\varepsilon) < \varepsilon) \lor (\rho(g)(\varepsilon) = \varepsilon \land  \rho(g)(\varepsilon+1) > \varepsilon + 1 ) \}$
    \item $Q_\varepsilon^{--} = \{g:(\rho(g)(\varepsilon) < \varepsilon) \lor (\rho(g)(\varepsilon) = \varepsilon \land  \rho(g)(\varepsilon+1) < \varepsilon + 1 ) \}$
\end{itemize}

\begin{theorem}[{\cite[Theorem 4.2]{rivas2010}}, \cite{deroin2016}]\label{thm:rivas}
$P_\varepsilon^+$ and $P_{\varepsilon}^-$ for $\varepsilon \in \R\setminus\Q$, $Q_\varepsilon^{++}$, $Q_\varepsilon^{+-}$, $Q_\varepsilon^{-+}$, and $Q_\varepsilon^{--}$ for $\varepsilon \in \Q$, and the 4 positive cones $P_{\infty}^{++}$, $P_{\infty}^{+-}$, $P_{\infty}^{-+}$, and $P_{\infty}^{--}$ corresponding to bi-orderings are all distinct and contain all the left-orderings on $\BS(1,n)$. 
\end{theorem}
Recall that for a left-orderable group $G$, a faithful action $D:G\to\Homeo_+(\mathbb{R})$ and a dense sequence $x_1,\dots \in \mathbb R$ one can recover an ordering $P_D(x_1,\dots)$ as mentioned after \cref{thm:ghys}. \cref{thm:rivas} tells us that by using the action $\rho$ of $\BS(1,n)$ on $\R$, one can classify all the bi-orderings by considering the first elements $\varepsilon=x_1$ of dense sequences in $\mathbb R$. Thus, given a positive cone $P$ of the form $P^+_\varepsilon$,
$P^-_\varepsilon$, $P^\circ_\infty$, or $Q_\varepsilon^\circ$ where $\circ\in \{++,--,+-,-+\}$ we refer
to $\varepsilon$ as the \emph{base point} of $P$ and $P^+$, $P^-$, $P^\circ$, and $Q^\circ$ as
its type.

We observe that given some $g\in \BS(1,n)$, $\varepsilon \in \R$, and $\circ\in \{+,-, ++, +-, -+, --\}$, we have $T_\varepsilon^\circ = (T_{g(\varepsilon)}^\circ)^{g^{-1}}$, where $T\in \{P,Q\}$. In particular, $\varepsilon$ is rational if and only if $g(\varepsilon)$ is rational and $\Q$ is a countable $\BS(1,n)$-invariant subset. Thus, the conjugation equivalence of the positive cones is Borel equivalent to the orbit equivalence relation of $\BS(1,n) \curvearrowright \R$. 

\begin{theorem}\label{thm:hypfinite}
The orbit equivalence relation $E$ generated by the affine action $\BS(1,n) = \langle a, b \mid b^{-1}ab = a^n\rangle \curvearrowright \R$ via $ a(x) = x+1$ and $b(x) = nx$ is hyperfinite. 
\end{theorem}

\begin{proof}
We will reduce $E$ to $E_t$, the tail equivalence relation, defined on $n^\omega$ by $AE_tB$ if $\exists p, q \forall k\ A(p+k) = B(q+k)$. This suffices as the tail equivalence relation is hyperfinite by \cite[Section 8]{DJK}. 

The reduction $f: \R \to n^\omega$ is given by sending $x\in \R$ to the base $n$ expansion of the fractional part $\{x\}$ of $x$. To show this is a reduction, let $x, y\in \R$. Assume first that $x\Elobs y$, so there is some $g\in \BS(1,n)$ such that $g(x) = y$. Assuming $g = a^r b^s$ such that $r \in \Z[1/n]$ and $s \in \Z$, we have $y = g(x) = a^r b^s(x) = n^{-s} x+r$. As $r\in \Z[1/n]$, we can multiply the equation by a power of $n$ to get $n^py  = n^q x + t$, where $p, q \in \N$ and $t \in \Z$. Since $t$ is an integer, we have $\{n^p y\} = \{n^q x\}$. However, in base $n$, $\{n^q x\}$ can be obtained from $\{x\}$ by truncating the first $q$ digits and shifting the decimal point by $q$ places, and similarly for $\{y\}$ and $\{n^p y\}$. Thus, $\{n^p y\} = \{n^q x\}$ implies that $\{x\} \E_t \{y\}$. 

Conversely, assume $\{x\} \E_t \{y\}$ in base $n$. Then there are some $q, p\in \N$ such that for every $k$, the $(q+k)$-th decimal place of $\{x\}$ is the same as the $(p+k)$-th decimal place of $\{y\}$. Thus, we have $\{n^q x\} = \{n^p y\}$, namely, there is some $t\in \Z$ with $n^py  = n^q x + t$, or equivalently $y = n^{q-p} x + t/n^p $. This means $y = g(x)$ with $g = a^{t/n^p}b^{p-q}$, so $x\E y$. This shows that $f$ is indeed a reduction. 
\end{proof}

It follows that $\Elobs$ is hyperfinite, and \cite{calderoni2023} showed that $\Elobs$ is not smooth.

\begin{corollary}
    $\Elobs \sim_B E_0$. 
\end{corollary}

\section{Computability of dynamical realizations}\label{sec:computability}

Given a left-ordering of $G$, it is straightforward to see that the left-ordering (considered as a relation on $G$) and the corresponding positive cone (considered as a subset of $G$) are Turing equivalent. It is thus natural to ask if this equivalence extends to the dynamical realization of the left-ordering. We will soon see that this is the case for $\BS(1,n)$.

Towards this fix an enumeration of the dyadic rational numbers $\mathbb
Q_2$ and recall that a real number $r\in \mathbb R$ is
\emph{left-c.e.}\ if its left cut is c.e., i.e., the set $\{ q<r : q\in \mathbb Q_2\}$ is computably enumerable. If both its left cut and right cut are c.e., then we say that it is \emph{computable}. 

\begin{proposition}\label{prop:dynequivcones}
Let $P$ be a left-ordering of $\BS(1,n)$, then $P$ is Turing equivalent to its
base point. 
Furthermore, the reductions are uniform
in the type. 
\end{proposition}

\begin{proof}
  We non-uniformly fix the type of $P$. We will assume that the type is $P^+$ with the construction easily adaptable to work for other types. We enumerate a right cut of its base point $\varepsilon$. For every
  $q\in \mathbb{Q}_2$ we enumerate $P$ and whenever we see $g=a^rb^s\in P$ such that
  $\rho(g)(q)=n^{-s}q+r>q$ we enumerate $q$ into $C$. Say $n^{-s}x
  +r>x$, then $x>\frac{-r}{n^{-s}-1}$, and so $\rho(g)(\varepsilon)>\varepsilon$ if
  and only if for every $q>\varepsilon$, $\rho(g)(q)>q$. Thus, $C_q$ is a right cut
  of $\varepsilon$. Similarly, we can enumerate a left cut. 

  Similarly, say we can compute a left cut $L$ and right cut $R$ of $\varepsilon$
  and are given a type $T$. We will give a proof assuming that $T=P^+$. The
  proof can be easily adapted to work for other types. For $q\in Q_2$, we can
  compute whether $g\in T_q$. By the affineness of $\rho$ we have that 
  \[ g\in P^+_\varepsilon \iff (\exists q\in R) \rho(g)(q)>q  \iff (\forall q\in R)
  \rho(g)(q)\geq q.\]
  Hence, we can compute $P^+_\varepsilon$ from its right cut.
\end{proof}

For $G$ a left-orderable computable group and $P$ a computable positive cone of
$G$ the following canonical index sets appear.
\begin{align*}
  I(G)&=\{ e:W_e\text{ is a positive cone}\}\\
  I(P,G)&=\{e: W_e\Elo P\}
\end{align*}
By definition the set $I(G)\in\Pi^0_2$ as membership in a c.e.\ set is
$\Sigma^0_1$. Similarly, it can be seen that $I(P,G)\in \Sigma^0_3$.

\begin{proposition}\label{prop:indexcones}
    Let $G$ be an infinite computable group with a computable left-ordering.
    Then $I(G)$ is $\Pi^0_2$-complete. 
\end{proposition}
\begin{proof}
  We reduce $\mathrm{Inf}$ to $I(G)$. Fix a computable positive cone $P$ and an
  index $e$ so that $W_e = P$. Given $n$, we build a total computable function
  $f$ so that $W_{f(n)}= P$ if and only if $n\in \mathrm{Inf}$. To do this, we
  enumerate $W_n$ in stages $W_{n,s}$, and, whenever $W_{n,s+1}\neq W_{n,s}$
  with $k=|W_{f(n),s+1}|$, we
  define $W_{f(n),s+1}=W_{e,k}$. The resulting set $W_{f(n)}$ is clearly c.e.\
  and $n\in \mathrm{Inf}$ if and only if $W_{f(n)}=W_e=P$.
\end{proof}
\begin{theorem}\label{thm:indexsets}
  \begin{enumerate}
    \item\label{it:bs1nbi} $I(P^\circ_\infty,\BS(1,n))$ is $\Pi^0_2$-complete for every $\circ\in
      \{++,--,+-,-+\}$.
    \item\label{it:bs1nlo} $I(P^\circ_\varepsilon,\BS(1,n))$ is
      $\Sigma^0_3$-complete for every computable $\varepsilon\in
      \mathbb{R}\setminus\mathbb{Q}$ and $\circ\in \{+,-\}$.
  \end{enumerate}
\end{theorem}
\begin{proof}
    The proof for \cref{it:bs1nbi} is analogous to the proof of
    \cref{prop:indexcones}.
    

    We prove the $\Sigma^0_3$-hardness of
    $I(P^\circ_\varepsilon, \BS(1,n))$ for a fixed computable $\varepsilon\in
    \mathbb R\setminus\mathbb{Q}$. That $I(P^\circ_\varepsilon,\BS(1,n))\in
    \Sigma^0_3$ follows easily from the definition. Given a $\Sigma^0_3$ set $S$
    we may assume that there is a computable function $g: \omega^2\to \omega$ such that
    \[ n\in S \iff \exists y W_{g(n,y)}\in \mathrm{Inf}.\]
    We may also assume that for every $k\in\omega$ there are infinitely many $y$ with $|W_{g(n,y)}|>k$. 
    
    Given $n$, we will define a left cut $C_{f(n)}$ in stages as follows.
    We let $C_{f(n)}^0=\emptyset$ and at every stage $s$ we choose some $y<s$ and
    extend $C_{f(n)}^s$ with the goal to make $C_{f(n)}$ a left cut of a real number
    $\delta_y$. Furthermore, every $y$ will define natural numbers
    $k_y$ and $l_y$ when it acts at a stage $s$.
    Every $\delta_y$ will differ from $\varepsilon$ at only finitely many digits in their base $n$ expansions and
    $\delta_0=\varepsilon$ with $k_0=l_0=0$. The construction is a classical finite injury construction where
    higher priority ``requirements'' (i.e., smaller $y$) will initialize the
    work of larger $y$. At the start of the construction, all $y$ are
    initialized.

    Assume we are at stage $s+1$ and that $y_0$ is least with
    $W_{g(x,y_0),s+1}\neq W_{g(x,y_0),s}$. Let $y_1$ be the $y$ that acted at
    stage $s$. If $y_0$ has been initialized since it last acted or has never acted before, do the following:
    
    Set $k_{y_0}=k_{y_1}+1=\langle p,q\rangle$. By induction we can assume
    that $\{n^r\delta_{y_1}\}=\{n^r\varepsilon\}$ for some $r\in \omega$. If $p\neq
      q$, then since $\varepsilon$ is irrational $n^p\delta_{y_1}\neq
      n^q\varepsilon$. Hence, there must be some least $i>r$ such that
      the $(p+i)$-th bit of $\{\delta_{y_1}\}$ is not equal to the $(q+i)$th bit
      of $\varepsilon$. We declare $\delta_{y_0}=\delta_{y_1}$ and $l_{y_0}=\max\{
      l_{y_1}, p+i\}$.

      Now, if $p=q$, then we consider the least $r>\max\{ p, l_{y_1}\}$ such
    that $\delta_{y_1}-n^{-r}\in (\max C_{f(n)}^s,\delta_{y_1})$ and the $r$-th decimal place of $\delta_{y_1}$ is nonzero, let
    $\delta_{y_0}=\delta_{y_1}-n^{-r}$ and set $l_{y_0}=r$. Note that the first $r-1$ decimal digits of $\delta_{y_0}$ and $\delta_{y_1}$ are the same. 

    At last, no matter if $y_0$ was initialized at the start of this stage or not, let
    \[C_{f(n)}^{s+1}=C_{f(n)}^s\cup \{ q_i\} \] where $q_i\in \mathbb Q_2$ is
    least in an enumeration of $\mathbb Q_2$ with $q_i<\delta_{y_0}$ and
    initialize all lower priority requirements. This finishes the construction.

    Let $C_{f(n)}=\lim_s C_{f(n)}^s$. We claim that this is a left cut. Indeed, if $q \in C_{f(n)}$ at some stage $s_0$, then we have $\delta_y[s] > \max(C_{f_n}^s) \ge q$ for $y\in \omega$ and $s > s_0$.\footnote{Following Lachlan
  $\delta_{y}[s]$ is the value of $\delta_y$ at stage $s$.} As infinitely often some $y$ will act, any $q' < q$ will eventually be enumerated into $C_{f(n)}$, making it a left cut. 
    
    If $n\in S$, let $y_0$ be the least such that
    $W_{g(n,y_0)}\in\mathrm{Inf}$. 
    Then there is a stage $s_0$ such that $y_0$ is never initialized again. 
    As every $y>y_0$ defines $\delta_{y}[s]\leq\delta_{y_0}$ at every $s>s_0$
    and $y_0$ acts infinitely often $\delta=\delta_{y_0}$ has left cut $C_{f(n)}$. Furthermore, as mentioned
    in the construction, $\{n^r\delta_{y_0}\}=\{n^r\varepsilon\}$ for some
    $r \in \omega$. From \cref{prop:dynequivcones} we get a positive cone $P^\circ_\delta$ and as can be seen in the proof of \cref{thm:hypfinite}, $P^\circ_\delta \Elobs P^\circ_\varepsilon$.
    
    On the other hand, suppose $n\notin S$, so that all $W_{g(n,y)}$ are finite. Given $y$, let
    $s_y$ be the stage after which $y$ is never initialized again. Then, by the
    construction $C_{f(n)}$ is the left cut for $\delta=\lim_y \delta_{y}[s_y]$
    and $\{ k_y[s_y]: y\in\omega\}=\omega$. Thus, there cannot be $p,q$ such that
    $n^p\delta=n^q\varepsilon$. This is the case because if $k_{y}[s_y]=\langle
    p,q\rangle$, then it is ensured at stage $s_y$ that $n^p\delta\neq n^q\varepsilon$
since $\{\delta\}\restrict l_y[s_y]=\{\delta_{y}[s_y]\}\restrict l_y[s_y]$
    and already this finite part witnessed that $n^p\delta\neq
    n^q\varepsilon$.
    Hence, $\{ \delta\}\not\E_t \{\varepsilon\}$ and for the positive cone $P^\circ_\delta$ we get from \cref{prop:dynequivcones}, $P^\circ_\delta\not\Elobs P^\circ_\varepsilon$.
  \end{proof}
  \begin{question}
    What is the complexity of $I(Q_\varepsilon^\circ, \BS(1,n))$ for
    $\varepsilon\in \mathbb Q$ and $\circ\in \{++,--,+-,-+\}$?
  \end{question}

%
%
%
  \bibliographystyle{plain}
  \bibliography{mybib}

\begin{thebibliography}{10}

\bibitem{baumslag1962}
Gilbert Baumslag and Donald Solitar.
\newblock Some two-generator one-relator non-{{Hopfian}} groups.
\newblock {\em Bulletin of the American Mathematical Society}, 68:199--201,
  1962.

\bibitem{bonanome18}
Marianna~C. Bonanome, Margaret~H. Dean, and Judith~Putnam Dean.
\newblock {\em A sampling of remarkable groups}.
\newblock Compact Textbooks in Mathematics. Birkh\"{a}user/Springer, Cham,
  2018.
\newblock Thompson's, self-similar, Lamplighter, and Baumslag-Solitar.

\bibitem{calderoni2022}
Filippo Calderoni and Adam Clay.
\newblock Borel structures on the space of left-orderings.
\newblock {\em Bulletin of the London Mathematical Society}, 54(1):83--94,
  2022.

\bibitem{calderoni2023a}
Filippo Calderoni and Adam Clay.
\newblock The {{Borel}} complexity of the space of left-orderings,
  low-dimensional topology, and dynamics.
\newblock {\em arXiv preprint arXiv:2305.03927}, 2023.

\bibitem{calderoni2023}
Filippo Calderoni and Adam Clay.
\newblock Condensation and left-orderable groups.
\newblock {\em arXiv preprint arXiv:2312.04993}, 2023.

\bibitem{CMMS2023}
Filippo Calderoni, David Marker, Luca Motto~Ros, and Assaf Shani.
\newblock Anti-classification results for groups acting freely on the line.
\newblock {\em Advances in Mathematics}, 418:Paper No. 108938, 45, 2023.

\bibitem{conley2023}
Clinton~T. Conley, Steve~C. Jackson, Andrew~S. Marks, Brandon~M. Seward, and
  Robin~D. {Tucker-Drob}.
\newblock Borel asymptotic dimension and hyperfinite equivalence relations.
\newblock {\em Duke Mathematical Journal}, 172(16):3175--3226, 2023.

\bibitem{darbinyan2020}
Arman Darbinyan.
\newblock Computability, orders, and solvable groups.
\newblock {\em The Journal of Symbolic Logic}, 85(4):1588--1598, 2020.

\bibitem{darbinyan2022}
Arman Darbinyan and Markus Steenbock.
\newblock Embeddings into left-orderable simple groups.
\newblock {\em Journal of the London Mathematical Society. Second Series},
  105(3):2011--2045, 2022.

\bibitem{deroin2016}
B.~Deroin, A.~Navas, and C.~Rivas.
\newblock Groups, {{Orders}}, and {{Dynamics}}.
\newblock http://arxiv.org/abs/1408.5805, October 2016.

\bibitem{DJK}
R.~Dougherty, S.~Jackson, and A.~S. Kechris.
\newblock The structure of hyperfinite {{Borel}} equivalence relations.
\newblock {\em Transactions of the American Mathematical Society},
  341(1):193--225, 1994.

\bibitem{downey1986}
Rodney~G. Downey and Stuart~A. Kurtz.
\newblock Recursion theory and ordered groups.
\newblock {\em Annals of Pure and Applied Logic}, 32:137--151, 1986.

\bibitem{feldman1977}
Jacob Feldman and Calvin~C. Moore.
\newblock Ergodic equivalence relations, cohomology, and von {{Neumann}}
  algebras. {{I}}.
\newblock {\em Transactions of the American Mathematical Society},
  234(2):289--324, 1977.

\bibitem{ghys2001}
{\'E}tienne Ghys.
\newblock Groups acting on the circle.
\newblock {\em Enseignement Mathematique}, 47(3/4):329--408, 2001.

\bibitem{harrington1990}
Leo~A. Harrington, Alexander~S. Kechris, and Alain Louveau.
\newblock A {{Glimm-Effros}} dichotomy for {{Borel}} equivalence relations.
\newblock {\em Journal of the American mathematical society}, pages 903--928,
  1990.

\bibitem{harrison-trainor2018}
Matthew {Harrison-Trainor}.
\newblock Left-orderable computable groups.
\newblock {\em The Journal of Symbolic Logic}, 83(1):237--255, 2018.

\bibitem{kechris2019}
Alexander Kechris.
\newblock The theory of countable {{Borel}} equivalence relations.
\newblock {\em preprint}, 2019.

\bibitem{kopytov1996}
Valeri\u{i}~M Kopytov and Nikola\u{i}~Y Medvedev.
\newblock {\em Right-ordered groups}.
\newblock Springer Science \& Business Media, 1996.

\bibitem{MMRT2019}
Dominique Malicet, Kathryn Mann, Crist\'{o}bal Rivas, and Michele Triestino.
\newblock Ping-pong configurations and circular orders on free groups.
\newblock {\em Groups Geom. Dyn.}, 13(4):1195--1218, 2019.

\bibitem{mannrivas2018}
Kathryn Mann and Crist\'{o}bal Rivas.
\newblock Group orderings, dynamics, and rigidity.
\newblock {\em Ann. Inst. Fourier (Grenoble)}, 68(4):1399--1445, 2018.

\bibitem{navas2010}
Andr{\'e}s Navas.
\newblock On the dynamics of (left) orderable groups.
\newblock {\em Universite de Grenoble. Annales de l'Institut Fourier},
  60(5):1685--1740, 2010.

\bibitem{poulin2024}
Antoine Poulin.
\newblock Borel complexity of the isomorphism relation of archimedean orders in
  finitely generated groups.
\newblock {\em arXiv preprint arXiv:2403.11326}, 2024.

\bibitem{rivas2010}
Crist{\'o}bal Rivas.
\newblock On spaces of {{Conradian}} group orderings.
\newblock {\em Journal of Group Theory}, 13(3):337--353, 2010.

\bibitem{rivas2012}
Crist{\'o}bal Rivas.
\newblock Left-orderings on free products of groups.
\newblock {\em Journal of Algebra}, 350:318--329, 2012.

\bibitem{RivasRomain16}
Crist\'{o}bal Rivas and Romain Tessera.
\newblock On the space of left-orderings of virtually solvable groups.
\newblock {\em Groups Geom. Dyn.}, 10(1):65--90, 2016.

\bibitem{sikora2004}
Adam~S. Sikora.
\newblock Topology on the spaces of orderings of groups.
\newblock {\em The Bulletin of the London Mathematical Society},
  36(4):519--526, 2004.

\bibitem{solomon2001}
Reed Solomon.
\newblock {$\Pi^1_1-\mathrm{CA}_0$} and order types of countable ordered
  groups.
\newblock {\em The Journal of Symbolic Logic}, 66(1):192--206, 2001.

\bibitem{solomon2002}
Reed Solomon.
\newblock {$\Pi_{1}^0$} classes and orderable groups.
\newblock {\em Annals of Pure and Applied Logic}, 115(1-3):279--302, 2002.

\end{thebibliography}
\end{document}